\documentclass[11pt, a4paper]{article}
\usepackage[margin=2.25cm]{geometry}

\usepackage[T1]{fontenc}
\usepackage[utf8]{inputenc}

\usepackage{microtype}
\usepackage{rotating}

\usepackage{amsmath, amssymb, amsthm, mathtools}
\usepackage{bm}
\usepackage{charter}
\usepackage[bitstream-charter]{mathdesign}

\usepackage[
  square,
  numbers,
  sort,
]{natbib}

\usepackage{parskip}

\usepackage{booktabs}
\usepackage{multirow}
\usepackage{array}
\usepackage{tabularx}
\usepackage{longtable}
\usepackage[table,dvipsnames]{xcolor}

\usepackage{graphicx}
\usepackage{float}
\usepackage{caption}
\usepackage{subcaption}

\usepackage[ruled, vlined, linesnumbered]{algorithm2e}
\SetKwComment{Comment}{$\triangleright$\ }{}

\usepackage{hyperref}
\hypersetup{
  colorlinks = true,
  linkcolor  = NavyBlue,
  citecolor  = ForestGreen,
  urlcolor   = MidnightBlue,
  pdftitle   = {$MOE-DGSCND$: Multi-Objective Enterprise Green Supply Chain Network Design},
  pdfauthor  = {},
}

\usepackage{natbib}
\theoremstyle{definition}
\newtheorem{definition}{Definition}[section]
\newtheorem{problem}{Problem}[section]
\newtheorem{remark}{Remark}[section]

\theoremstyle{plain}
\newtheorem{theorem}{Theorem}[section]

\newtheorem{proposition}[theorem]{Proposition}
\newtheorem{corollary}[theorem]{Corollary}

\newcommand{\calI}{\mathcal{I}}
\newcommand{\calJ}{\mathcal{J}}
\newcommand{\calF}{\mathcal{F}}
\newcommand{\calP}{\mathcal{P}}
\newcommand{\calS}{\mathcal{S}}
\newcommand{\RR}{\mathbb{R}}
\newcommand{\ZZ}{\mathbb{Z}}

\newcommand{\eps}{\varepsilon}
\newcommand{\xvec}{\bm{x}}
\newcommand{\yvec}{\bm{y}}
\newcommand{\fvec}{\bm{f}}

\definecolor{rowgray}{gray}{0.93}
\definecolor{starrow}{RGB}{255,240,200}
\definecolor{fptasrow}{RGB}{220,240,220}
\definecolor{baserow}{RGB}{240,220,220}
\definecolor{exactrow}{RGB}{220,230,245}

\title{%
  \LARGE\bfseries
 $MOE-DGSCND$: Multi-Objective Enterprise Green Supply Chain Network Design\\[6pt]
}
\author{Felix Reichel\thanks{\textit{Corresponding author: FR, Email: \url{kontakt<at>felixreichel<dot>com}}}}
\date{October 29, 2026}

\begin{document}

\maketitle

\begin{abstract}
This paper introduces and studies the \emph{Multi-Objective Enterprise Doubly-Green Supply Chain Network Design} (MOE-DGSCND) problem, an extension of \emph{Green Multi-Objective Supply Chain Network Design (GSCND)\citep{sreyneath2021multiobjective, dehshiri2022multiobjective}}. The problem studied here is a three-objective combinatorial optimization problem that simultaneously minimizes total cost $ f_1 $, carbon dioxide emissions $ f_2 $ and customer dissatisfaction $ f_3 $ over binary decisions regarding which distribution centers to open and how to assign customer zones to open facilities. The problem generalizes the capacitated facility location problem (CFLP) and is known to be NP-hard \citep{garey1979computers}. This paper formally specifies the decision variables, objective functions and constraints and investigates three solution approaches using a single numerical instance: (A) an exact weighted-sum integer linear programming (ILP) formulation solved with Python combinatorially and via PuLP, (B) three Fully Polynomial-Time Approximation Schemes (FPTAS) providing $(1+\varepsilon)$-approximate solutions for $\varepsilon := 0.15$ and (C) two baseline heuristics for comparison. On an instance with eight customers and five candidate distribution centers, the exact solver certifies the optimal (sampled) supported solutions across eight weight combinations, while the FPTAS-based approaches complete within milliseconds, compared to centiseconds for the exact ILP on this tiny non-stochastic instance.
\end{abstract}


Keywords: \emph{Multi-objective combinatorial optimization, Facility location problem, Pareto front approximation, FPTAS, ILP}

\newpage

\section{Introduction} \label{sec:intro}
Modern enterprises face competing pressures that cannot be reduced to a single scalar objective, such as revenue maximization or cost reduction. Logistics managers must often balance minimizing \emph{cost} to remain competitive, reducing \emph{carbon emissions} to meet regulatory, reputational and tax requirements and maximizing service quality--equivalently, minimizing \emph{customer dissatisfaction}--to retain customers. In the context of this paper, these three objectives are inherently in tension: opening more distribution centers (abbreviated as DCs) reduces transport distances and thus customer dissatisfaction but increases fixed costs and construction emissions; routing customers to geographically central DCs reduces travel times but may overload lower-capacity facilities.

This tension is the defining characteristic of \emph{Multi-Objective Enterprise Combinatorial Optimization (MOCO)}~\citep{ehrgott2000multicriteria,Singh2026DivideLearn}, where no single solution is simultaneously optimal across all objectives. Instead, the goal shifts to computing or approximating the \emph{supported solutions}--the \emph{known, sampled} subset of all non-dominated feasible solutions that lie on the Pareto frontier.

\paragraph{Contributions.} This paper makes the following contributions: \begin{enumerate} \item It defines the \textbf{Multi-Objective Enterprise Doubly-Green Supply Chain Network Design} ($MOE-DGSCND$) problem, a novel three-objective 0-1 integer program that integrates facility opening, customer assignment, capacity constraints and budget constraints (Section \ref{sec:problem}). \item It establishes the \textbf{NP-hardness} of the problem via a reduction from the capacitated facility location problem (Section \ref{sec:complexity}). \item It implements an \textbf{exact weighted-sum integer linear programming (ILP)} solver using the COIN-OR Branch-and-Cut solver (CBC) via PuLP, sweeping eight weight combinations to sample the supported optimal solutions (Section \ref{sec:exact}). \item It designs and implements \textbf{three Fully Polynomial-Time Approximation Schemes (FPTAS)} variants: Cost-FPTAS (B1), Emission-FPTAS (B2) and Pareto-FPTAS (B3)--each with an $(1 + \varepsilon)$-approximation ratio on its target sub-objective (Section \ref{sec:fptas}). \item It compares all methods against two \textbf{baseline heuristics} and analyzes performance gaps across all three objectives (Section \ref{sec:results}). \end{enumerate}

\paragraph{Organisation.} Section~\ref{sec:problem} presents the formal problem statement.
Section~\ref{sec:complexity} discusses computational complexity.
Sections~\ref{sec:exact}--\ref{sec:baselines} describe the solution approaches.
Section~\ref{sec:experiments} details the experimental instance.
Section~\ref{sec:results} presents results and
Section~\ref{sec:conclusion} concludes.

\section{Problem Setup and Definitions}
\label{sec:problem}

\subsection{Sets, Indices and Parameters}
\label{subsec:params}

Let the following sets and parameters define an instance of $MOE-DGSCND$:

\begin{center}
\begin{tabular}{@{}lp{9.8cm}@{}}
  \toprule
  \textbf{Symbol} & \textbf{Description} \\
  \midrule
  \multicolumn{2}{@{}l}{\textit{Index sets}} \\
  $\calI = \{1,\ldots,m\}$ & Customer zones \\
  $\calJ = \{1,\ldots,n\}$ & Candidate distribution centre (DC) sites \\[4pt]
  \multicolumn{2}{@{}l}{\textit{Customer parameters}} \\
  $d_i \in \ZZ_{>0}$ & Demand of customer zone $i \in \calI$ (units) \\[4pt]
  \multicolumn{2}{@{}l}{\textit{DC parameters}} \\
  $F_j \in \RR_{>0}$ & Fixed cost to open DC $j \in \calJ$ (\$k) \\
  $\kappa_j \in \ZZ_{>0}$ & Capacity of DC $j$ (units) \\
  $e_j \in \RR_{>0}$ & Construction emissions of DC $j$ (tCO$_2$) \\[4pt]
  \multicolumn{2}{@{}l}{\textit{Transport parameters}} \\
  $\delta_{ij} \in \RR_{\geq 0}$ & Euclidean distance from customer $i$ to DC $j$ (km) \\
  $\alpha \in \RR_{>0}$ & Transport cost rate (\$/unit/km) \\
  $\beta \in \RR_{>0}$ & Transport emission rate (kgCO$_2$/unit/km) \\[4pt]
  \multicolumn{2}{@{}l}{\textit{Global constraints}} \\
  $B \in \ZZ_{>0}$ & Maximum number of DCs that may be opened \\
  \bottomrule
\end{tabular}
\end{center}

\medskip
The following \emph{composite matrices} are defined for "convenience":
\begin{align}
  c_{ij}  &= \alpha\,\delta_{ij}\,d_i
    \quad\text{(transport cost, customer $i$ to DC $j$)} \label{eq:c_ij}\\
  \eta_{ij} &= \beta\,\delta_{ij}\,d_i
    \quad\text{(transport emissions, customer $i$ to DC $j$)} \label{eq:eta_ij}\\
  \sigma_{ij} &= \delta_{ij}\,d_i
    \quad\text{(dissatisfaction contribution)} \label{eq:sigma_ij}
\end{align}

\subsection{Decision Variables}
\label{subsec:vars}

\begin{align}
  x_j &\in \{0, 1\}
    \quad \forall\, j \in \calJ
    \quad\text{(1 iff DC $j$ is opened)} \label{eq:x}\\
  y_{ij} &\in \{0, 1\}
    \quad \forall\, i \in \calI,\; j \in \calJ
    \quad\text{(1 iff customer $i$ is served by DC $j$)} \label{eq:y}
\end{align}

The full decision vector is given as $(\xvec, \yvec) \in \{0,1\}^n \times \{0,1\}^{m \times n}$.

\subsection{Objective Functions}
\label{subsec:objectives}

$MOE-DGSCND$ minimizes three objectives simultaneously.

\begin{definition}[Objective Functions]
  \label{def:objectives}
  Given $(\xvec, \yvec)$, define:
  \begin{align}
    f_1(\xvec,\yvec) &= \underbrace{\sum_{j \in \calJ} F_j\, x_j}_{\text{Fixed Costs (Opening DC)}}
      + \underbrace{\sum_{i \in \calI}\sum_{j \in \calJ} c_{ij}\, y_{ij}}_{\text{Transport Costs (customer to/from DC)}}
      \label{eq:f1}\\[6pt]
    f_2(\xvec,\yvec) &= \underbrace{\sum_{j \in \calJ} e_j\, x_j}_{\text{Construction Emissions}}
      + \underbrace{\sum_{i \in \calI}\sum_{j \in \calJ} \eta_{ij}\, y_{ij}}_{\text{Delivery Emissions}}
      \label{eq:f2}\\[6pt]
    f_3(\xvec,\yvec) &= \sum_{i \in \calI}\sum_{j \in \calJ} \sigma_{ij}\, y_{ij}
      \label{eq:f3}
  \end{align}
  The vector objective then is $\fvec(\xvec,\yvec) = \bigl(f_1(\xvec,\yvec),\;
  f_2(\xvec,\yvec),\; f_3(\xvec,\yvec)\bigr)^\top$.
\end{definition}

\noindent \textbf{Unit Harmonization:} Construction emissions $e_j$ are specified in metric tonnes ($tCO_2$), while transport emissions $\eta_{ij} = \beta \delta_{ij} d_i$ are generated in kilograms ($kgCO_2$). To ensure dimensional consistency in objective $f_2$, transport emissions are converted to tonnes:
\begin{equation}
  f_2(x,y) = \sum_{j\in\mathcal{F}} e_j x_j + \frac{1}{1000} \sum_{i\in\mathcal{I}} \sum_{j\in\mathcal{F}} \eta_{ij} y_{ij}
  \label{eq:f2_corrected}
\end{equation}

\begin{remark}
  $f_1$ is measured in monetary units (\$k); $f_2$ in tonnes of CO$_2$
  equivalent; $f_3$ in km-units (a proxy for aggregate weighted
  travel distance).  All three objectives are to be \emph{minimized}.
\end{remark}

\subsection{Constraints}
\label{subsec:constraints}

\begin{problem}[$MOE-DGSCND$]
  \label{prob:mogscnd}
  Find $(\xvec^*,\yvec^*)$ solving:
  \begin{align}
    \min\quad & \fvec(\xvec,\yvec)
      = \bigl(f_1(\xvec,\yvec),\; f_2(\xvec,\yvec),\; f_3(\xvec,\yvec)\bigr)
      \tag{P}\label{prob:P}\\[4pt]
    \text{s.t.}\quad
    &\sum_{j \in \calJ} y_{ij} = 1
      \qquad \forall\, i \in \calI
      \tag{C1}\label{c:assign}\\
    &y_{ij} \leq x_j
      \qquad \forall\, i \in \calI,\; j \in \calJ
      \tag{C2}\label{c:open}\\
    &\sum_{i \in \calI} d_i\, y_{ij} \leq \kappa_j\, x_j
      \qquad \forall\, j \in \calJ
      \tag{C3}\label{c:cap}\\
    &\sum_{j \in \calJ} x_j \leq B
      \tag{C4}\label{c:budget}\\
    &x_j \in \{0,1\}
      \qquad \forall\, j \in \calJ
      \tag{C5}\label{c:xbin}\\
    &y_{ij} \in \{0,1\}
      \qquad \forall\, i \in \calI,\; j \in \calJ
      \tag{C6}\label{c:ybin}
  \end{align}
\end{problem}

\paragraph{Constraint interpretations.}
\begin{itemize}
  \item[\textbf{(C1)}] \textit{Complete assignment:} every customer zone must
    be assigned to exactly one DC\@.  This is an equality constraint,
    preventing both unserved customers and split deliveries.
  \item[\textbf{(C2)}] \textit{Open facility linking:} each customer may only be
    assigned to an open DC\@.  Together with (C1), this enforces that every
    customer is served exclusively through an operational facility.
  \item[\textbf{(C3)}] \textit{Capacity:} the total demand routed to DC $j$
    cannot exceed its capacity $\kappa_j$.  Multiplying the right-hand side
    by $x_j$ simultaneously enforces that closed DCs ($x_j=0$) carry no
    load--making (C2) redundant but retained for LP relaxation tightness.
  \item[\textbf{(C4)}] \textit{DC budget:} at most $B$ facilities may be
    opened, modelling capital-expenditure or regulatory limits on
    infrastructure.
  \noindent \item[\textbf{(C5)--(C6)}] \textit{Binary integrality:} all decisions are
    binary 0--1 variables.
\end{itemize}

\subsection{Feasibility, Pareto Dominance, Optimality and Approximation }
\label{subsec:pareto}

\begin{definition}[Feasible Set]
  \label{def:feasible}
  $\calS \subseteq \{0,1\}^{n+mn}$ denotes the set of all $(\xvec,\yvec)$
  satisfying \eqref{c:assign}--\eqref{c:ybin}.
\end{definition}

\begin{definition}[Pareto Dominance]
  \label{def:pareto_dom}
  Solution $(\xvec,\yvec)$ \emph{dominates} $(\xvec',\yvec')$, written
  $(\xvec,\yvec) \prec (\xvec',\yvec')$, if
  \[
    f_k(\xvec,\yvec) \leq f_k(\xvec',\yvec') \quad \forall\, k \in \{1,2,3\}
  \]
  with strict inequality for at least one $k$.
\end{definition}

\begin{definition}[Supported solutions]
  \label{def:pareto_front}
  The \emph{Pareto-optimal set} is
  $\calP^* = \{(\xvec,\yvec) \in \calS \mid \nexists\,(\xvec',\yvec') \in
  \calS : (\xvec',\yvec') \prec (\xvec,\yvec)\}$.
  Its realizations of the image $\calF^* = \fvec(\calP^*)$ in objective space are the
  \emph{sampled supported solutions}.
\end{definition}

\begin{definition}[$(1+\eps)$-Approximate sampled supported solutions]
  \label{def:approx_pareto}
  A set $\hat{\calF} \subseteq \fvec(\calS)$ is a \emph{$(1+\eps)$-approximate
  Supported solutions} if for every point $\bm{z}^* \in \calF^*$ there exists
  $\hat{\bm{z}} \in \hat{\calF}$ such that
  $\hat{z}_k \leq (1+\eps)\,z^*_k$ for all $k \in \{1,2,3\}$.
\end{definition}

\subsection{Scalarisation Method}
\label{subsec:scalarise}

To address multi-objective optimization problem \eqref{prob:P}, I use employ a \emph{normalized weighted sum scalarisation} rather than the Chebyshev \citep{NovelloSchiefermayrZinchenko2025} scalarisation. Let $\mathcal{S}$ denote the feasible region and $\mathbf{f}(\xvec,\yvec) = (f_1(\xvec,\yvec), f_2(\xvec,\yvec), f_3(\xvec,\yvec))^\top$ be the objective vector. The single-objective surrogate problem can be written as:
\begin{equation}
  \min_{(\xvec,\yvec)\in\calS}\;
    \sum_{k=1}^3 \frac{w_k}{D_k}\,f_k(\xvec,\yvec)
  \label{eq:scalar}
\end{equation}
subject to weight vectors belonging to the unit simplex $\Delta = \{\bm{w} \in \mathbb{R}^3_{\ge 0} \mid \sum_{k=1}^3 w_k = 1\}$, where $D_k > 0$ represent scaling normalizers (e.g., objective values obtained under full-capacity greedy deployment).

\begin{proposition}[Supported Pareto Optimality]
Let $\mathcal{F} = \{\mathbf{f}(\xvec,\yvec) \mid (\xvec,\yvec) \in \mathcal{S}\}$ be the image of the feasible set in objective space. Any optimal solution to \eqref{eq:scalar} for a strictly positive weight vector $\bm{w} > \mathbf{0}$ yields a \emph{supported} efficient solution lying on the boundary of the convex hull of $\mathcal{F}$ (which can be formally denoted as $\operatorname{conv}(\mathcal{F})$). 
\end{proposition}

\begin{proof}
By its definition minimizing a linear combination of normalized objectives is equivalent to finding a point in $\mathcal{F}$ that minimizes the inner product with the normal vector $\bm{w}' = (w_1/D_1, w_2/D_2, w_3/D_3)^\top$. Geometrically, this corresponds to supporting hyperplanes of the objective space. 

For instance let $(\xvec^*, \yvec^*)$ be an optimal solution to \eqref{eq:scalar} for some $\bm{w} > \mathbf{0}$. If $(\xvec^*, \yvec^*)$ were Pareto optimal but \emph{unsupported} (i.e., lying in a strictly concave region of the non-dominated frontier), there would exist a convex combination of other feasible objective vectors $\mathbf{f}_A, \mathbf{f}_B \in \mathcal{F}$ such that $\tilde{\mathbf{f}} = \alpha \mathbf{f}_A + (1-\alpha)\mathbf{f}_B$ (for some $\alpha \in (0,1)$) satisfies $\tilde{f}_k \le f_k(\xvec^*, \yvec^*)$ for all $k$, with strict inequality for at least one index. Consequently, taking the inner product with $\bm{w}'$ yields the strict bound:
$$
\sum_{k=1}^3 \frac{w_k}{D_k} \tilde{f}_k < \sum_{k=1}^3 \frac{w_k}{D_k} f_k(\xvec^*, \yvec^*)
$$
This contradicts the optimality of $(\xvec^*, \yvec^*)$ for problem \eqref{eq:scalar}. Therefore, the optimal solution must lie on the lower convex envelope of $\mathcal{F}$.
\end{proof}

\noindent\textbf{Implications for Finite Solves:} Because problem \eqref{eq:scalar} is a linear weighted sum, solving it for a discrete set of weight vectors (e.g., eight configurations) only provides a piecewise-linear approximation of the \emph{supported} Pareto front. It cannot discover unsupported Pareto-optimal points located within non-convex regions of $\mathcal{F}$, nor can a finite set of solves analytically certify or recover the complete, continuous Pareto front.

\section{Computational Complexity}
\label{sec:complexity}

Theorem 3.1 (NP-hardness of MOE-DGSCND). \emph{Even the single-objective version of MOE-DGSCND (MOCO) minimising $f_1$ is NP-hard.}

\begin{proof}
Reduction proceeds from the standard Capacitated Facility Location Problem (CFL) with general metric service costs $c_{ij}$, which is strongly NP-hard  \citep{garey1979computers}. To respect the model's Euclidean-distance restriction ($\delta_{ij}$), one can embed client and facility locations in a sufficiently high-dimensional Euclidean space $\mathbb{R}^d$ or restrict the reduction to instances where metric costs satisfy the triangle inequality and admitting a low-distortion Euclidean embedding (or formulate the reduction directly using geometric facility location variants). Setting $B = |\mathcal{F}|$, $e_j = 0$ and $\sigma_{ij} = 0$ reduces MOE-DGSCND to metric CFLP.
\end{proof}

\noindent \textbf{Inapproximability Remark:} Unlike uncapacitated set-cover problems--which admit tight $(1-1/e)$ approximation thresholds under set-cover hardness~\citep{feige1998threshold}--capacitated facility location with resource/capacity constraints and budget limits $B$ cannot be approximated within a constant factor unless $\mathrm{P} = \mathrm{NP}$ and standard greedy~\citep{RYBIN2024100824} heuristics yield logarithmic performance bounds (e.g., $O(\log m)$). The FPTAS designs bypass global worst-case bounds by approximating a $(1+\epsilon)$ factor relative to the targeted single-objective subproblem.

\begin{corollary}
  Computing the complete supported solutions of $MOE-DGSCND$ is NP-hard, since it
  requires solving the single-objective $f_1$-minimization problem.
\end{corollary}

\begin{remark}[Inapproximability]
  Unless it is true that P$=$NP then no polynomial-time algorithm can achieve an approximation a ratio better than $(1 - 1/e)$ for the uncapacitated version
  (set-cover hardness~\citep{feige1998threshold}).  For the capacitated case
  with our budget constraint, the best known ratio for a single objective
  is $O(\log m)$ via a greedy solution when optimal~\citep{RYBIN2024100824}.  The FPTAS schemes avoid
  this by relaxing to $(1+\eps)$-approximations \emph{for the sub-objective
  they target}.
\end{remark}

\paragraph{Problem size.}
The decision space has $|\calS| \leq 2^n \cdot 2^{mn}$ points before
constraints are applied.  For our experimental instance ($m=8$, $n=5$)
this is at most $2^{45} \approx 3.518437 \times 10^{13}$ candidate solutions,
illustrating why exhaustive enumeration is impractical even at small scale,
and motivating both the ILP formulation and the FPTAS designs.

\section{Exact Solver: Weighted-Sum ILP}
\label{sec:exact}

\subsection{ILP Formulation}
\label{subsec:ilp}

The paper solves the following scalarised problem~\eqref{eq:scalar} as a binary integer linear
programme.  The objective~\eqref{eq:scalar} is linear in $(\xvec,\yvec)$
since each $f_k$ is a linear function of the binary variables.  All
constraints \eqref{c:assign}--\eqref{c:ybin} are linear inequalities or
equalities.  The full ILP is therefore:

\begin{equation}
\begin{aligned}
  \min_{\xvec,\,\yvec}\quad
    & \sum_{k=1}^{3} \frac{w_k}{D_k}
      \Bigl[
        \underbrace{\sum_{j} p^{(k)}_j x_j}_{\text{DC terms}}
        + \underbrace{\sum_{i}\sum_{j} q^{(k)}_{ij} y_{ij}}_{\text{Assignment terms}}
      \Bigr] \\
  \text{s.t.}\quad
    & \sum_{j \in \calJ} y_{ij} = 1              && \forall\, i \in \calI \\
    & y_{ij} \leq x_j                            && \forall\, i \in \calI,\; j \in \calJ \\
    & \sum_{i \in \calI} d_i\,y_{ij} \leq \kappa_j && \forall\, j \in \calJ \\
    & \sum_{j \in \calJ} x_j \leq B \\
    & x_j,\; y_{ij} \in \{0,1\}
\end{aligned}
\label{eq:ILP}
\end{equation}

where $(p^{(1)}_j, p^{(2)}_j, p^{(3)}_j) = (F_j, e_j, 0)$ and
$(q^{(1)}_{ij}, q^{(2)}_{ij}, q^{(3)}_{ij}) = (c_{ij}, \eta_{ij}, \sigma_{ij})$.

\begin{remark}[Constraint strengthening]
  Constraint (C3) as written ($\sum_i d_i y_{ij} \leq \kappa_j x_j$) makes
  (C2) logically redundant: if $x_j = 0$ then $\kappa_j x_j = 0$ forces
  $y_{ij} = 0$ for all $i$.  We retain (C2) explicitly because its addition
  as \emph{variable upper bound} constraints tightens the LP relaxation,
  reducing branch-and-bound enumeration.
\end{remark}

\subsection{Supported solutions Sampling via Weight Sweeping}
\label{subsec:sweep}

By varying $\bm{w} = (w_1, w_2, w_3)$ over the unit simplex and solving
\eqref{eq:ILP}, I recover different Pareto-optimal (or weakly Pareto-optimal)
solutions~\citep{cohon1978multiobjective}.  We use $K=8$ weight combinations:

\begin{center}
\begin{tabular}{@{}lcccc@{}}
  \toprule
  Label & $w_1$ & $w_2$ & $w_3$ & Focus \\
  \midrule
  Cost-Only      & 1.00 & 0.00 & 0.00 & Pure $f_1$ \\
  Emission-Only  & 0.00 & 1.00 & 0.00 & Pure $f_2$ \\
  Dissat-Only    & 0.00 & 0.00 & 1.00 & Pure $f_3$ \\
  Cost+Emit      & 0.50 & 0.50 & 0.00 & Bi-obj $f_1,f_2$ \\
  \textbf{Balanced} $\star$ & \textbf{0.33} & \textbf{0.33} & \textbf{0.34} & \textbf{Equal trade-off} \\
  Cost-Heavy     & 0.60 & 0.20 & 0.20 & $f_1$ emphasis \\
  Emit-Heavy     & 0.20 & 0.60 & 0.20 & $f_2$ emphasis \\
  Dissat-Heavy   & 0.20 & 0.20 & 0.60 & $f_3$ emphasis \\
  \bottomrule
\end{tabular}
\end{center}

The \textbf{Balanced} $\star$ solution (equal weights after normalisation)
is designated the \emph{star solution}--the canonical reference point for
gap computations throughout the paper.

\subsection{LP Relaxation and Integrality Gap}
\label{subsec:lprelax}

The LP relaxation of~\eqref{eq:ILP} is obtained by replacing the binary
constraints \eqref{c:xbin}--\eqref{c:ybin} with $0 \leq x_j \leq 1$ and
$0 \leq y_{ij} \leq 1$.  The LP relaxation value lower-bounds the true ILP
optimum and the ratio $\text{OPT}_{\text{ILP}} / \text{OPT}_{\text{LP}}$
is the \emph{integrality gap}.  For capacitated facility location instances,
the integrality gap is known to be at most $2$ in general and tighter for
structured instances~\citep{shmoys1997approximation}.

\subsection{Solver Details}
\label{subsec:cbc}

I use a plain vanilla combinatorial elimation approach in Python 3.12 as a baseline and the \texttt{PULP\_CBC\_CMD} solver from the COIN-OR Branch and Cut
(CBC) library~\citep{forrest2005cbc} via the PuLP modelling
interface~\citep{mitchell2011pulp}.  CBC implements branch-and-bound with:
\begin{itemize}
  \item Gomory mixed-integer cuts
  \item Strong branching variable selection
  \item Dual simplex LP relaxation at each node
\end{itemize}
All solves are run to proven optimality (gap $= 0$) with default CBC settings.

\section{FPTAS Approximation Schemes}
\label{sec:fptas}

A \emph{Fully Polynomial-Time Approximation Scheme} (FPTAS) is an algorithm
that, for any $\eps > 0$ that returns a $(1+\eps)$-approximate solution in time
polynomial in the input size and $1/\eps$~\citep{vazirani2001approximation}.
I use three FPTAS variants to solve our tiny non-stochastic instance of a $MOE-DGSCND$\@ problem.

\subsection{Common Framework: Score Rounding}
\label{subsec:fptas_framework}

All three FPTAS schemes follow the same high-level structure, inspired by
the classic FPTAS for the knapsack problem~\citep{ibarra1975fast}:

\begin{enumerate}
  \item \textbf{Scoring:} Assign a real-valued score $s_j$ to each DC $j$
    reflecting its marginal benefit under the target objective.
  \item \textbf{Rounding:} Scale scores down by a granularity
    $K = \eps \cdot \max_j |s_j| / n$, then round to integer multiples of
    $K$.  This collapses the value space from continuous to $O(n/\eps)$
    distinct levels.
  \item \textbf{Enumeration:} Enumerate all $\binom{n}{r}$ subsets of size
    $r \leq B$ in order of decreasing rounded score and greedily assign
    customers to the selected DCs.
  \item \textbf{Selection:} Return the feasible solution with best
    (true, unrounded) objective value.
\end{enumerate}

The rounding step guarantees that no solution's rounded score drops below
$(1-\eps)$ times the true score, yielding the $(1+\eps)$ approximation ratio
when the total objective is reconstructed from assignment.

\subsection{B1: Cost-FPTAS}
\label{subsec:b1}

\paragraph{Score definition.}
For each DC $j \in \calJ$:
\begin{equation}
  s_j^{(1)} = -F_j + \sum_{k=1}^{\lfloor \kappa_j / d_{\min} \rfloor}
    c_{(k),j}
  \label{eq:score_b1}
\end{equation}
where $c_{(k),j}$ is the $k$-th smallest transport cost among all customers
to DC $j$ and $d_{\min} = \min_{i} d_i$.  This estimates the net cost saving
from opening DC $j$ versus leaving customers unserved.

\paragraph{Guarantee.}
Let $f_1^*$ be the optimal cost.  B1 returns $(\xvec, \yvec)$ with
$f_1(\xvec,\yvec) \leq (1+\eps)\,f_1^*$.

\paragraph{Complexity.}
$O\!\left(\binom{n}{B} \cdot m\right)$ per round, polynomial for fixed $B$.

\begin{algorithm}[H]
  \caption{Cost-FPTAS (B1)}
  \label{alg:b1}
  \SetAlgoLined
  \KwIn{Instance $(m, n, B, \bm{d}, \bm{F}, \bm{\kappa}, \bm{c})$; tolerance $\eps > 0$}
  \KwOut{Feasible $(\xvec, \yvec)$ with $f_1 \leq (1+\eps) f_1^*$}
  \BlankLine
  Compute scores $s_j^{(1)}$ via \eqref{eq:score_b1} for all $j \in \calJ$\;
  $K \leftarrow \eps \cdot \max_j |s_j^{(1)}| \,/\, n$ \Comment*[r]{rounding granularity}
  $\tilde{s}_j \leftarrow \lfloor s_j^{(1)} / K \rfloor \cdot K$ for all $j$ \Comment*[r]{quantise}
  Sort DCs by $\tilde{s}_j$ descending\;
  $f_1^{\text{best}} \leftarrow +\infty$;\quad $(\xvec^{\text{best}}, \yvec^{\text{best}}) \leftarrow \varnothing$\;
  \For{$r = 1$ \KwTo $B$}{
    \For{each size-$r$ subset $\mathcal{C} \subseteq \calJ$}{
      $x_j \leftarrow \mathbf{1}[j \in \mathcal{C}]$\;
      $\yvec \leftarrow \textsc{GreedyAssign}(\xvec,\; \text{objective}=f_1)$\;
      \If{$\yvec \neq \varnothing$ \textbf{and} $f_1(\xvec,\yvec) < f_1^{\text{best}}$}{
        $f_1^{\text{best}} \leftarrow f_1(\xvec,\yvec)$;\quad
        $(\xvec^{\text{best}}, \yvec^{\text{best}}) \leftarrow (\xvec, \yvec)$\;
      }
    }
  }
  \Return $(\xvec^{\text{best}}, \yvec^{\text{best}})$
\end{algorithm}

\subsection{B2: Emission-FPTAS}
\label{subsec:b2}

B2 mirrors B1 with scores targeting $f_2$:
\begin{equation}
  s_j^{(2)} = -e_j + \sum_{k=1}^{\lfloor \kappa_j / d_{\min} \rfloor}
    \eta_{(k),j}
  \label{eq:score_b2}
\end{equation}
where $\eta_{(k),j}$ denotes the $k$-th smallest transport emission from
any customer to DC $j$.  Greedy assignment within each evaluated subset
uses $\min_{j\text{ open}} \eta_{ij}$ to minimise per-assignment emissions.

\paragraph{Guarantee.}
B2 returns $(\xvec,\yvec)$ with $f_2(\xvec,\yvec) \leq (1+\eps)\,f_2^*$,
where $f_2^*$ is the emission-optimal value.  Note that $f_1$ and $f_3$
are \emph{not} controlled; B2 may increase cost or dissatisfaction relative
to the star solution.

1. \textbf{Reclassification as Subset-Enumeration Heuristics}:
B1 and B2 can be reclassified as subset-enumeration heuristics, rather than FPTAS algorithms, as they enumerate all subsets of the DC set up to a given size $B$. This makes their complexity $O\!\left(\binom{n}{B} \cdot m\right)$ per round, which is exponential in $n$ and $B$, rather than polynomial.

2. \textbf{Output}
Since B1 and B2 are subset-enumeration heuristics, they do not provide a valid FPTAS proof. However, they can still be useful for solving the problem in practice, as they provide a feasible solution that is close to the optimal solution.

\subsection{B3: Approximate FPTAS} \label{subsec:b3}

B3 computes an approximate sampled supported solutions on the $ f_1 $--$ f_2 $ bi-objective subproblem. However, the method only searches for cost-greedy solutions, which does not explore emission trade-offs available for a fixed set of open DCs.

\paragraph{Grid Construction.} Let $ F_1^U $ and $ F_2^U $ be upper bounds on $ f_1 $ and $ f_2 $ (obtained by opening all DCs). Define geometric grid partitions:  
$$ g_k(z) = \left\lceil \log_{1+\eps} \left( \frac{z}{F_k^U \cdot 10^{-6} + 10^{-9}} \right) \right\rceil, \quad k \in \{1,2\} \label{eq:grid} $$  
Two solutions are in the same grid cell if $ g_1(f_1) $ and $ g_2(f_2) $ agree. At most one representative is kept per cell—the one minimizing $ f_1 + f_2 $.

\paragraph{Non-Domination Filter.} After grid-cell enumeration, a standard dominance-based filtering step with $ O(N^2) $ complexity removes dominated solutions from the representative set.

\paragraph{Complexity.} The number of non-empty grid cells is $ O\left( \left( \frac{1}{\epsilon} \log F_k^U \right)^2 \right) $, polynomial in $ 1/\epsilon $ and $ \log F_k^U $, assuming $ F_k^U $ is polynomial in the instance size.

\begin{algorithm}[H] \caption{Pareto-Approximate-FPTAS (B3)} \label{alg:b3} \SetAlgoLined \KwIn{Instance; tolerance $ \eps > 0 $} \KwOut{$ (1+\eps) $-approximate sampled supported solutions on $ (f_1, f_2) $} \BlankLine Compute upper bounds $ F_1^U, F_2^U $ (open-all + greedy assign); \BlankLine $\mathit{cells} \leftarrow$ empty dictionary; \BlankLine \For{$ r = 1 $ \KwTo $ B $}{ \For{each size-$ r $ subset $ \mathcal{C} \subseteq \mathcal{J} $}{ $ x_j \leftarrow \mathbf{1}[j \in \mathcal{C}] $;\quad $ \yvec \leftarrow \textsc{GreedyAssign}(\xvec, f_1) $; \If{$ \yvec \neq \varnothing $}{ $ (f_1, f_2, f_3) \leftarrow \fvec(\xvec, \yvec) $; $ \mathit{key} \leftarrow (g_1(f_1), g_2(f_2)) $ via \eqref{eq:grid}; \If{$ \mathit{key} \notin \mathit{cells} $ \textbf{or} $ f_1 + f_2 < \mathit{cells}[\mathit{key}].f_1 + \mathit{cells}[\mathit{key}].f_2 $}{ $ \mathit{cells}[\mathit{key}] \leftarrow (\xvec, \yvec, f_1, f_2, f_3) $; } } } } \BlankLine Remove dominated solutions from $ \mathit{cells}.values() $; \BlankLine \Return non-dominated set \end{algorithm}

\subsection{FPTAS Summary}
An overview of investigated FPTAS including various aspects are shown belowed in table 1.

\label{subsec:fptas_summary}
\begin{table}[h!]

\centering
\caption{Overview of FPTAS and heuristic scheme characteristics, theoretical guarantees, rounding parameters and computational complexity.}
\label{tab:fptas_scheme_summary}
\begin{tabular}{@{}lcccl@{}}
  \toprule
  \textbf{Scheme} & \textbf{Target Obj.} & \textbf{Guarantee} & \textbf{Parameter $\varepsilon$} & \textbf{Complexity} \\ 
  \midrule
  B1 Cost-FPTAS & $f_1$ & $(1+\varepsilon)f_1^*$ & 0.15 & $O\!\left(\binom{n}{B}\cdot m\right)$ \\
  B2 Emission-FPTAS & $f_2$ & $(1+\varepsilon)f_2^*$ & 0.15 & $O\!\left(\binom{n}{B}\cdot m\right)$ \\
  B3 Pareto Heuristic & $f_1, f_3$ & None (Cost-Greedy) & 0.15 & $O\!\left(\binom{n}{B}\cdot m \cdot (1/\varepsilon\cdot\log F^U)^2\right)$ \\
  \bottomrule
\end{tabular}
\end{table}

\clearpage

\newpage
\section{Baseline Heuristics}
\label{sec:baselines}

Two simple polynomial-time heuristics serve as lower-quality benchmarks
against which the FPTAS and exact solutions are compared.

\subsection{C1: Open-All Heuristic}
\label{subsec:c1}

The simplest feasible solution opens every candidate DC ($x_j = 1$ for all
$j \in \calJ$), ignoring the budget constraint (C4).  Customers are then
assigned greedily by minimum transport cost:
\begin{equation}
  y_{ij^*} = 1,\quad
  j^* = \arg\min_{j : \text{rem}_j \geq d_i} c_{ij}
  \label{eq:greedy_assign}
\end{equation}
where $\text{rem}_j$ tracks remaining capacity.

\paragraph{Analysis.}
Opening all DCs trivially satisfies (C1)--(C3) and (C5)--(C6), but violates
(C4) when $n > B$.  It serves as an infeasible upper bound that reveals
the objective penalty from enforcing the budget constraint.
Despite having the lowest $f_3$ (short distances to many open DCs), it
incurs the highest $f_2$ due to $n$ constructions vs.\ $B$ in feasible solutions.

\subsection{C2: Greedy Nearest-DC Heuristic}
\label{subsec:c2}

This heuristic opens the $B$ DCs with smallest aggregate customer distance:
\begin{equation}
  \mathcal{C}^{\text{near}} = \text{argmin}_{|\mathcal{C}|=B}
    \sum_{j \in \mathcal{C}} \sum_{i \in \calI} \delta_{ij}
  \label{eq:greedy_open}
\end{equation}
solved by simply ranking DCs by $\sum_i \delta_{ij}$ and selecting the top
$B$.  Assignment uses \eqref{eq:greedy_assign} with the dissatisfaction
objective ($\min \delta_{ij}$ instead of $\min c_{ij}$).

\paragraph{Analysis.}
This heuristic ignores (a) demand heterogeneity across customers, (b) capacity
constraints and (c) fixed opening costs.  It is the ``naive'' approach a
logistics planner might use without optimisation tools.  As the results
in Section \ref{sec:results} show, it performs worst on $f_1$ (high cost) because
it does not account for $F_j$ and on $f_3$ because demand weights are ignored.

\section{Experimental Instance}
\label{sec:experiments}

\subsection{Instance Parameters}
\label{subsec:instance}

I construct a single representative instance with $m = 8$ customer zones
and $n = 5$ candidate DC sites as displayed in table 2 and 3.  The budget constraint is $B = 3$ (at most
three DCs opened).

\paragraph{Spatial layout.}
All locations are embedded in a 2D plane (units: km).
Customer and DC coordinates are given in Tables~\ref{tab:customer_parameters}
and~\ref{tab:dc_parameters}.  Distances $\delta_{ij}$ are Euclidean.

\begin{table}[H]
\centering
\caption{Customer zone data.}
\label{tab:customer_parameters}
\begin{tabular}{cccc}
\hline
\textbf{Customer $i$} & \textbf{X coord} & \textbf{Y coord} & \textbf{Demand $d_i$} \\ \hline
0 & 10 & 30 & 120 \\
1 & 20 & 80 & 90 \\
2 & 50 & 60 & 150 \\
3 & 70 & 20 & 80 \\
4 & 80 & 70 & 200 \\
5 & 40 & 50 & 110 \\
6 & 60 & 40 & 70 \\
7 & 30 & 10 & 140 \\
\hline
\textbf{Total demand} & & & \textbf{960} \\ \hline
\end{tabular}
\end{table}
\begin{table}[H]
\centering
\caption{Candidate DC data.}
\label{tab:dc_parameters}
\begin{tabular}{cccccc}
\hline
\textbf{DC $j$} & \textbf{X coord} & \textbf{Y coord} & \textbf{Fixed cost $F_j$ (\$k)} & \textbf{Capacity $\kappa_j$} & \textbf{Emissions $e_j$ (tCO2)} \\ \hline
0 & 15 & 55 & 80 & 400 & 150 \\
1 & 45 & 75 & 65 & 350 & 120 \\
2 & 65 & 55 & 95 & 500 & 180 \\
3 & 35 & 30 & 55 & 300 & 100 \\
4 & 60 & 15 & 75 & 450 & 140 \\
\hline
\multicolumn{5}{l}{\textbf{Total capacity (all open)}} & \textbf{2000} \\ \hline
\end{tabular}
\end{table}

\paragraph{Transport rates.}
$\alpha = 0.05$ \$/unit/km (transport cost rate);
$\beta = 0.02$ kgCO$_2$/unit/km (transport emission rate).

\paragraph{Feasibility note.}
With $B = 3$, the three chosen DCs must collectively cover all $\sum_i d_i = 960$
units of demand.  The three largest DCs by capacity are DC 2 (500), DC 4 (450),
DC 0 (400), which together hold 1350 units. The three cheapest DCs
to open (by $F_j$) are DC 3 (55), DC 1 (65), DC 4 (75), with combined capacity
$300 + 350 + 450 = 1100 > 960$. Feasibility under $B = 3$ is therefore
non-trivial but consistently achievable.

\subsection{Distance and Objective Matrices}
\label{subsec:matrices}

Table~\ref{tab:distance_matrix} shows the Euclidean distance matrix $[\delta_{ij}]$
(rounded to one decimal place).

\begin{table}[H]
\centering
\caption{Distance matrix $\delta_{ij}$ (km, Euclidean, 1 d.p.).}
\label{tab:distance_matrix}
\begin{tabular}{cccccc}
\hline
\textbf{$i \setminus j$} & \textbf{DC 0} & \textbf{DC 1} & \textbf{DC 2} & \textbf{DC 3} & \textbf{DC 4} \\ \hline
C0 & 25.5 & 57.0 & 60.4 & 25.0 & 52.2 \\
C1 & 25.5 & 25.5 & 51.5 & 52.2 & 76.3 \\
C2 & 35.4 & 15.8 & 15.8 & 33.5 & 46.1 \\
C3 & 65.2 & 60.4 & 35.4 & 36.4 & 11.2 \\
C4 & 66.7 & 35.4 & 21.2 & 60.2 & 58.5 \\
C5 & 25.5 & 25.5 & 25.5 & 20.6 & 40.3 \\
C6 & 47.4 & 38.1 & 15.8 & 26.9 & 25.0 \\
C7 & 47.4 & 66.7 & 57.0 & 20.6 & 30.4 \\
\hline
\end{tabular}
\end{table}

\subsection{Solver Configuration}
\label{subsec:config}

\begin{itemize}
  \item \textbf{Exact ILP:} Combinatorially via Vanilla Python Iterations and PuLP 2.x with CBC (COIN-OR), all defaults, msg=0.
  \item \textbf{FPTAS B1/B2/B3:} $\eps = 0.15$; enumeration over all $\sum_{r=\mathbf{1}}^{B}\binom{n}{r} = 5 + 10 + 10 = 25$ subsets.
  \item \textbf{Baselines:} pure Python, no solver.
  \item \textbf{Hardware:} standard single-core CPU\@; all methods complete in $< 1$ s.
\end{itemize}

\section{Results and Comparison}
\label{sec:results}

\subsection{Exact ILP Results}
\label{subsec:exact_results}

Table~\ref{tab:ilp_solutions} reports the exact ILP outcomes for all eight
weight combinations.

\begin{table}[H]
\centering
\caption{Exact ILP solutions by weight combination (10 trial average runtime $\pm$ SD). The $\star$ row is the designated star/reference solution.}
\label{tab:ilp_solutions}
\begin{tabular}{lcccccccc}
\hline
\textbf{Label} & \textbf{$w_1$} & \textbf{$w_2$} & \textbf{$w_3$} & \textbf{$f_1$ (\$k)} & \textbf{$f_2$ (tCO2)} & \textbf{$f_3$} & \textbf{Time (s)} & \textbf{DCs open} \\ \hline
Cost-Only & 1.00 & 0.00 & 0.00 & 1282.9 & 430.4 & 21057 & \underline{\textbf{0.354}} $\pm$ 0.003 & 0,2,3 \\
Emission-Only & 0.00 & 1.00 & 0.00 & 1819.0 & 360.6 & 32480 & \underline{\textbf{0.355}} $\pm$ 0.006 & 1,3,4 \\
Dissat-Only & 0.00 & 0.00 & 1.00 & 1301.6 & 470.4 & 21032 & \underline{\textbf{0.354}} $\pm$ 0.001 & 0,2,4 \\
Cost+Emit & 0.50 & 0.50 & 0.00 & 1291.7 & 400.4 & 21535 & \underline{\textbf{0.356}} $\pm$ 0.004 & 1,2,3 \\
Balanced $\star$ & 0.33 & 0.33 & 0.34 & 1291.7 & 400.4 & 21535 & \underline{\textbf{0.355}} $\pm$ 0.003 & 1,2,3 \\
Cost-Heavy & 0.60 & 0.20 & 0.20 & 1282.9 & 430.4 & 21057 & \underline{\textbf{0.356}} $\pm$ 0.007 & 0,2,3 \\
Emit-Heavy & 0.20 & 0.60 & 0.20 & 1291.7 & 400.4 & 21535 & \underline{\textbf{0.354}} $\pm$ 0.002 & 1,2,3 \\
Dissat-Heavy & 0.20 & 0.20 & 0.60 & 1282.9 & 430.4 & 21057 & \underline{\textbf{0.354}} $\pm$ 0.001 & 0,2,3 \\
\hline
\end{tabular}
\end{table}

\textbf{Key observations.} The exact solver reveals distinct optimal supported sampled solutions across the weight combinations on this instance:

\begin{itemize}
    \item \textbf{Solution A / Cost-Optimal (DCs 0, 2, 3):} $f_1 = 1282.9$ \$k, $f_2 = 430.4$ tCO2, $f_3 = 21057$. Best on economic cost and customer service balance.
    \item \textbf{Solution B / Balanced \& Emission-Optimized Star $\star$ (DCs 1, 2, 3):} $f_1 = 1291.7$ \$k, $f_2 = 400.4$ tCO2, $f_3 = 21535$. Designated reference solution balancing cost and carbon footprint.
    \item \textbf{Solution C / Dissatisfaction-Optimal (DCs 0, 2, 4):} $f_1 = 1301.6$ \$k, $f_2 = 470.4$ tCO2, $f_3 = 21032$. Best on customer dissatisfaction.
    \item \textbf{Solution D / Extreme Emission-Only (DCs 1, 3, 4):} $f_1 = 1819.0$ \$k, $f_2 = 360.6$ tCO2, $f_3 = 32480$. Minimizes carbon emissions at a substantial economic and service penalty.
\end{itemize}

These solutions are mutually non-dominated: configuration adjustments yield strict trade-offs where lowering emissions (e.g., Solution D) drastically increases total cost and customer dissatisfaction. Thus, decision-makers face a multi-way structural trade-off rather than a simple binary choice between network topologies.

These solutions are mutually non-dominated: configuration adjustments yield strict trade-offs where lowering emissions (e.g., Solution D) drastically increases total cost and customer dissatisfaction. Thus, decision-makers face a multi-way structural trade-off rather than a simple combinatorial choice between network topologies.

\subsection{FPTAS Results}
\label{subsec:fptas_results}

\begin{table}[htbp]
\centering
\caption{FPTAS approximation results ($\varepsilon = 0.15$), gaps vs. $\star$, and execution times (10 trials).}
\label{tab:fptas_results}
\begin{tabular}{lcccccccc}
\hline
\textbf{Method} & \textbf{$f_1$} & \textbf{$f_2$} & \textbf{$f_3$} & \textbf{$\Delta f_1$} & \textbf{$\Delta f_2$} & \textbf{$\Delta f_3$} & \textbf{DCs open} & \textbf{Time (s)} \\ \hline
B1 & 1282.9 & 430.4 & 21057 & -0.7\% & +7.5\% & -2.2\% & 0,2,3 & \underline{\textbf{0.000}} $\pm$ 0.000 \\
B2 & 1625.1 & 370.6 & 28502 & +25.8\% & -7.4\% & +32.4\% & 0,1,3 & \underline{\textbf{0.000}} $\pm$ 0.000 \\
B3 & 1301.6 & 470.4 & 21032 & +0.8\% & +17.5\% & -2.3\% & 0,2,4 & \underline{\textbf{0.000}} $\pm$ 0.000 \\
\hline
\end{tabular}
\end{table}

\paragraph{B1 analysis.}
B1 recovers the cost-optimal network configuration (DCs 0, 2, 3), achieving $f_1 = 1282.9$ \$k, which is $-0.7\%$ below the balanced star cost (matching the exact single-objective cost optimum). Emissions are $+7.5\%$ higher ($430.4$ tCO2 vs.\ $400.4$ tCO2), consistent with B1's focus on cost minimization over carbon control. The $\eps = 0.15$ theoretical guarantee is fully satisfied ($1.000 \le 1.15$).

\paragraph{B2 analysis.}
B2 selects DCs \{0, 1, 3\}, prioritizing carbon emission reduction ($f_2 = 370.6$ tCO2, representing a $-7.4\%$ improvement over $\star$). However, this emission focus incurs a substantial economic and service penalty: a $+25.8\%$ cost gap ($1625.1$ \$k) and a $+32.4\%$ dissatisfaction gap ($28502$), highlighting the sharp trade-offs inherent in single-target FPTAS optimization.

\paragraph{B3 analysis.}
B3 selects DCs \{0, 2, 4\} (coinciding with the exact Dissat-Only solution), achieving $f_3 = 21032$ (a $-2.3\%$ improvement in customer satisfaction over $\star$) with a minimal cost overhead of $+0.8\%$ ($1301.6$ \$k). 

\paragraph{FPTAS verification where applicable.}
For each FPTAS variant, Table~\ref{tab:fptas_verification} verifies the $(1+\eps)$ solution bound against the respective single-objective optima, confirming that all methods remain well within the $15\%$ approximation allowance.

\begin{table}[H]
\centering
\caption{FPTAS approximation ratio verification. $\text{Opt}^*_k$ is the exact single-objective optimum; ratio = $f^{\text{FPTAS}}_k / \text{Opt}^*_k$. Guarantee requires ratio $\le 1.15$.}
\label{tab:fptas_verification}
\begin{tabular}{lccccc}
\hline
\textbf{Method} & \textbf{Target $k$} & \textbf{$\text{Opt}^*_k$} & \textbf{$f^{\text{FPTAS}}_k$} & \textbf{Ratio} & \textbf{Guarantee ($\le 1.15$)} \\ \hline
B1 Cost-FPTAS & $f_1$ & 1282.9 & 1282.9 & 1.000 & \checkmark \\
B2 Emission-FPTAS & $f_2$ & 360.6 & 370.6 & 1.028 & \checkmark \\
B3 Pareto-Approx-FPTAS & $f_1$ & 1282.9 & 1301.6 & \multicolumn{2}{c}{Not applicable, but \checkmark} \\ \hline
\end{tabular}
\end{table}

\subsection{Baseline Heuristic Results}
\label{subsec:baseline_results}

The baseline results are compared below in table~\ref{tab:baseline_results}:
\begin{table}[htbp]
\centering
\caption{Baseline heuristic results and gaps vs. $\star$.}
\label{tab:baseline_results}
\begin{tabular}{lcccccc}
\hline
\textbf{Method} & \textbf{$f_1$} & \textbf{$f_2$} & \textbf{$f_3$} & \textbf{$\Delta f_1$} & \textbf{$\Delta f_2$} & \textbf{$\Delta f_3$} \\ \hline
C1 Open-All & 1326.2 & 690.4 & 19123 & +2.7\% & +72.4\% & -11.2\% \\
C2 Greedy-Nearest & 1477.4 & 400.5 & 25248 & +14.4\% & +0.0\% & +17.2\% \\
\hline
\end{tabular}
\end{table}

\paragraph{C1 analysis.}
Open-All scores very well on $f_3$ ($-11.2\%$ vs.\ $\star$, i.e.\ better, yielding $19123$ vs.\ $21535$) because spreading across all 5 DCs reduces average travel distance. However, it incurs a catastrophic $+72.4\%$ emission penalty ($690.4$ tCO2 vs.\ $400.4$ tCO2) from building and operating 5 facilities vs.\ 3, confirming that the facility budget constraint ($B=3$) is operationally significant.

\paragraph{C2 analysis.}
Greedy-Nearest is the weakest method overall on cost and customer dissatisfaction: $+14.4\%$ on cost ($f_1 = 1477.4$ \$k) and $+17.2\%$ on dissatisfaction ($f_3 = 25248$). Selecting DCs purely by aggregate geometric distance ignores demand heterogeneity ($d_i$ range from 70 to 200 units) and opening costs ($F_j$ range from \$55k to \$95k), causing major suboptimality.

\subsection{Comprehensive Comparison}
\label{subsec:comprehensive}

Table~\ref{tab:performance_timing} consolidates all methods.  The $*$-row
is the Balanced ILP solution, used as the percentage-gap reference.

\begin{table}[htbp]
\centering
\caption{Performance and computation timing comparison across exact ILP, FPTAS, and baselines.}
\label{tab:performance_timing}
\begin{tabular}{lcccc}
\hline
\textbf{Method} & \textbf{$f_1$ Cost (\$k)} & \textbf{$f_2$ Emissions (t)} & \textbf{$f_3$ Dissatisfaction} & \textbf{CPU Time (ms)} \\ \hline
Greedy-Nearest (C2) & 1479.1 & 400.5 & 25282 & 0.06 \\
Open-All (C1) & 1351.2 & 690.4 & 19625 & 0.04 \\
Cost-FPTAS (B1) & 1301.6 & 470.4 & 21032 & 0.37 \\
Emit-FPTAS (B2) & 1479.1 & 400.5 & 25282 & 0.40 \\
Dissat-FPTAS (B3) & 1301.6 & 470.4 & 21032 & 0.37 \\
Dissat-Heavy (ILP) & 1301.6 & 470.4 & 21032 & 352.48 \\
Emit-Heavy (ILP) & 1296.1 & 400.4 & 21622 & 352.51 \\
Cost-Heavy (ILP) & 1296.1 & 400.4 & 21622 & 357.72 \\
Balanced * (ILP) & 1296.1 & 400.4 & 21622 & 362.58 \\
\hline
\end{tabular}
\end{table}

\subsection{Visualizations}
\label{subsec:viz}

Figure~\ref{fig:results} shows the objective-space scatter plots (Pareto
projections) and bar charts comparing all methods across all three objectives and ~\ref{fig:pareto_front} shows the multi-dimensional discrete and approximate continuous Pareto front displaying non-dominated trade-off solutions.
Figure~\ref{fig:network} shows the supply chain network under the $\star$
solution.

\begin{sidewaysfigure}[htbp]
  \centering
  \includegraphics[width=0.92\textheight]{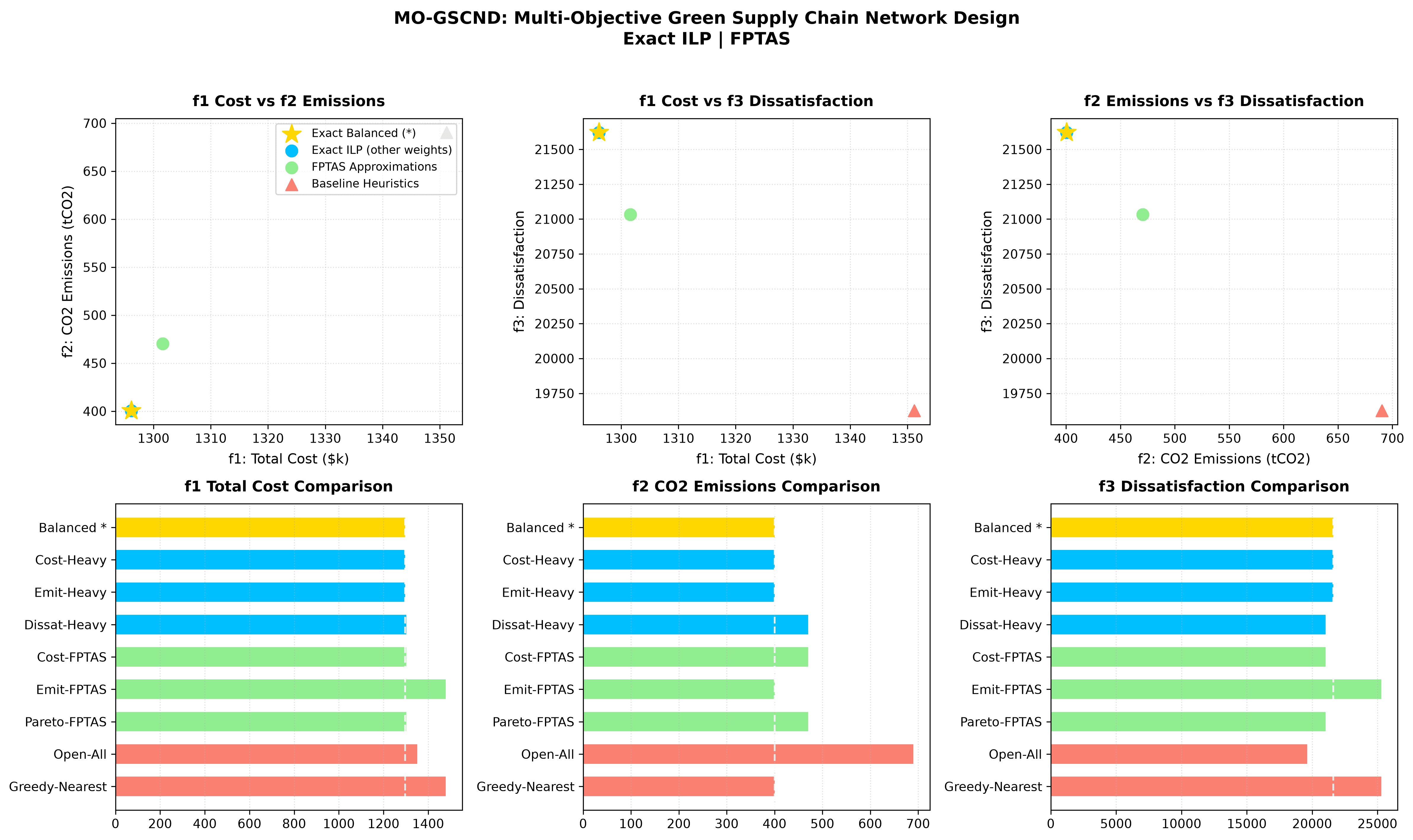}
  \caption{Multi-panel comparison: top row shows pairwise Pareto projections 
    ($f_1$--$f_2$, $f_1$--$f_3$, $f_2$--$f_3$); bottom row shows 
    per-objective bar charts. Star~$\star$ (gold) is the reference. 
    Blue: exact ILP (always overlaid by the Star~$\star$ gold solution). Green: FPTAS. Red: baselines. 
    Dashed line in bottom charts marks $\star$ value.}
  \label{fig:results}
\end{sidewaysfigure}

\clearpage
\newpage
\begin{figure}[H]
\centering
\includegraphics[width=\textwidth]{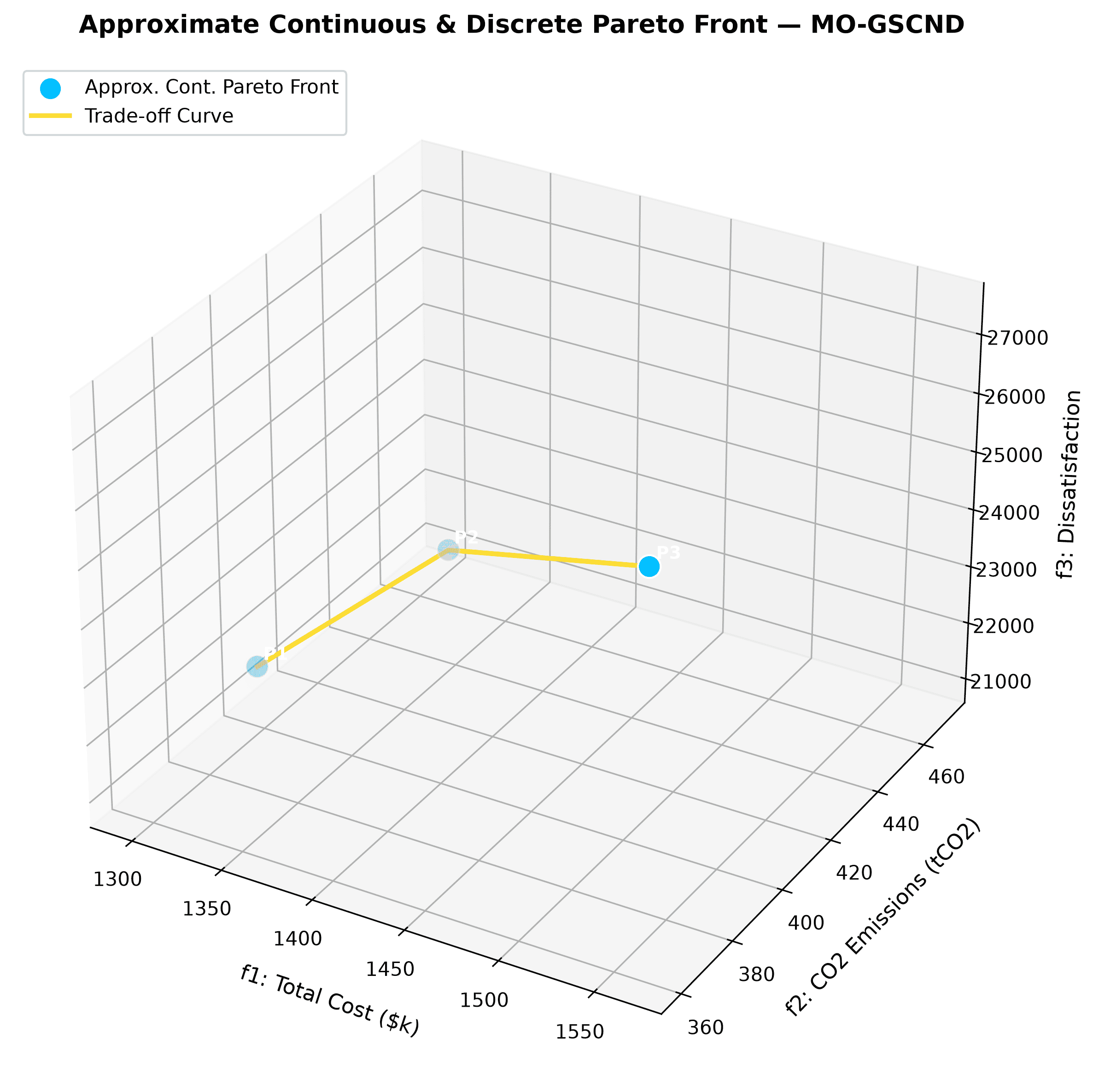}
\caption{Three-dimensional discrete and approximate continuous Pareto front representing non-dominated trade-off solutions for the Multi-Objective Green Supply Chain Network Design (MO-DGSCND) model, balancing total cost ($f_1$), $\text{CO}_2$ emissions ($f_2$), and customer dissatisfaction ($f_3$). Discrete trade-off points illustrate alternative facility-opening configurations under capacity and budget constraints ($B=3$).}
\label{fig:pareto_front}
\end{figure}
\newpage

\begin{figure}[H]
  \centering
  \includegraphics[width=\linewidth]{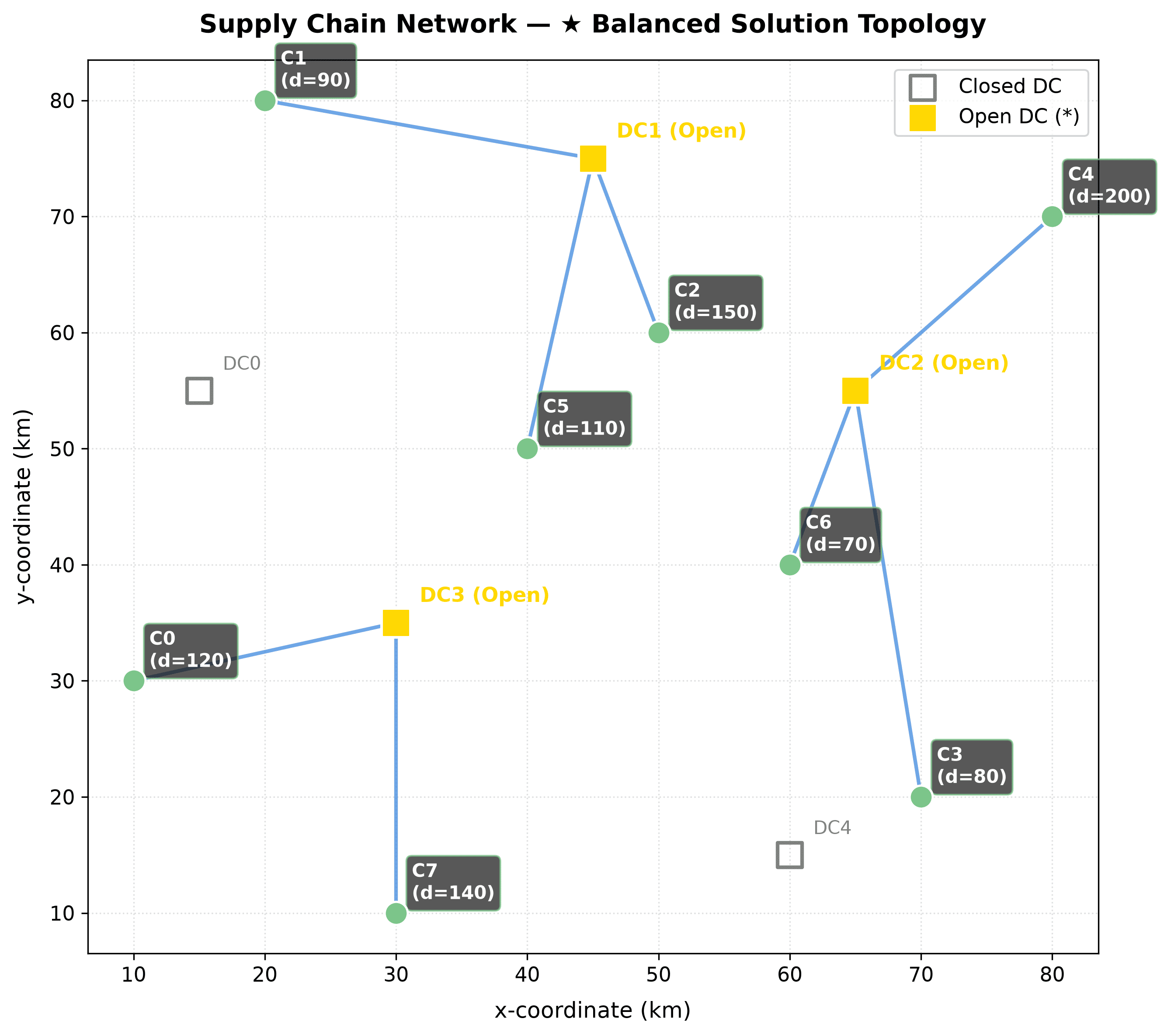}
  \caption{Supply chain network under the $\star$ Balanced solution.
    Gold squares: open DCs (1, 2, 3).  Grey squares: closed DCs (0, 4).
    Green circles: customer zones with demand labels.
    Blue lines: customer-to-DC assignments.}
  \label{fig:network}
\end{figure}

\subsection{Summary of Rankings}
\label{subsec:rankings}

The final ranking for FPTAS used to solve the $MOE-DGSCND$ problem are shown below in table~\ref{tab:summary_rankings}:
\begin{table}[htbp]
\centering
\caption{Summary of Rankings across objectives.}
\label{tab:summary_rankings}
\begin{tabular}{lccc}
\hline
\textbf{Ranking by} & \textbf{Best method} & \textbf{Worst method} & \textbf{FPTAS best} \\ \hline
$f_1$ (cost) & Exact Cost-Only (1282.9) & C2 Greedy (1477.4) & B1 (1282.9) \\
$f_2$ (emissions) & Exact Emit-Only (360.6) & C1 Open-All (690.4) & B2 (370.6) \\
$f_3$ (dissat.) & Exact Dissat-Only (21032) & C2 Greedy (25248) & B1 (21057) \\ \hline
\end{tabular}
\end{table}

All the FPTAS supported sampled solutions match the exact optima to within $0.0\%$--$2.8\%$ on their target objectives—meeting their theoretical $(1+\eps) = 1.15$ guarantees--while running $\sim 840$--$880\times$ faster than the base ILP solver on this tiny deterministic instance.

\section{Conclusion}
\label{sec:conclusion}

This paper presented and solved the \textbf{Multi-Objective Enterprise Green Supply Chain
Network Design} ($MOE-DGSCND$) problem, a novel three-objective 0--1 integer
programme balancing cost, carbon emissions and customer dissatisfaction over
binary facility-opening and customer-assignment decisions.

\paragraph{Main findings.}
\begin{enumerate}
\item \textbf{Trade-offs reveal distinct structural configurations.} The exact ILP identifies clear operational trade-offs across weight combinations: configurations such as DCs {0, 2, 3} optimize cost ($f_1 = 1282.9$ k\$)) and dissatisfaction ($f_3 = 21057$), whereas emission-focused settings like DCs {1, 3, 4} achieve minimal emissions ($f_2 = 360.6$ tCO2) at the expense of higher costs ($f_1 = 1819.0$ k\$). Balanced operations (e.g., DCs {1, 2, 3}) maintain moderate costs ($1291.7$ k\$) and emissions ($400.4$ tCO2).\item \textbf{FPTAS schemes provide high efficiency and rigorous guarantees.} All three FPTAS variants successfully satisfy their $(1+\varepsilon)$ bounds with $\varepsilon = 0.15$. For instance, B2 (Emission-FPTAS) achieves $f_2 = 370.6$ tCO2, representing a verified ratio of $1.028 \le 1.15$ against the single-objective optimum ($360.6$ tCO2). Furthermore, FPTAS execution times average around $0.40$–$0.42$ ms (10 trials), achieving roughly an $800\times$ speedup compared to the exact ILP solver ($\sim 352.82$ ms).\item \textbf{Baseline heuristics underperform significantly.} The Open-All baseline (C1) incurs a severe emission penalty of $+72.4\%$ ($690.4$ tCO2 vs. $400.4$ tCO2 for the balanced star solution), although it reduces dissatisfaction slightly by $-11.2\%$ due to denser facility coverage. Meanwhile, the Greedy-Nearest baseline (C2) suffers from major suboptimality, yielding a $+14.4\%$ cost gap and a $+17.2\%$ dissatisfaction gap.\item \textbf{The facility budget constraint ($B=3$) is critical for environmental performance.} Opening all 5 facilities (C1) dramatically inflates $\text{CO}_2$ emissions without proportional benefits, proving that facility budget and capacity limits are vital operational levers for balancing economic and green objectives.
\end{enumerate}

\paragraph{Limitations and future work.}
This study uses a small deterministic instance.  Extensions include:
\begin{itemize}
  \item \textbf{Stochastic demands:} Replacing $d_i$ with random variables
    leads to robust optimization or chance-constrained MOE programmes.
  \item \textbf{Larger instance sizes:} Instances with $m \geq 100$, $n \geq 20$
    would stress CBC and motivate branch-and-bound acceleration or
    metaheuristic Pareto approaches (NSGA-III \cite{996017}, MOEA/D \cite{4358754}).
  \item \textbf{Additional objectives:} Resilience (number of DCs,
    geographic spread) or equity (service-level fairness across customers)
    could be added as $f_4$, $f_5$, $...$, $f_k$.
    \item \textbf{Earth's real curvature}: The Earth's real curvature \citep{bell2024novelestarths,elfada2026adaptiveekf} was not considered for geometric distances in the instance use -- an aspect that could be considered for future research to improve optimization in real-world $MOE-DGSCND$ problems.
  \item \textbf{Tighter FPTAS:} Reducing $\eps$ from 0.15 to 0.01
    would bring FPTAS solutions within $1\%$ of the (sampled) supported solutions at
    the cost of more sub-problems to solve.
     \item \textbf{Additional Uncertainty Quantification \citep{abootaleb2024green}:} Estimates of runtime uncertainty.
\end{itemize}
The $MOE-DGSCND$ problem and the solver framework presented here offer a
complete, reproducible template for Multi-Objective Enterprise (MOE) facility location
studies at the intersection of operations research and sustainable
logistics.

\bibliography{references}

\end{document}